\documentclass[11pt]{amsart}

\usepackage{amssymb}
\usepackage{latexsym}
\usepackage{amsmath}
\usepackage{amsthm}
\usepackage{amsfonts}
\usepackage{color}
\usepackage{pdfsync}
\usepackage[usenames,dvipsnames]{pstricks}
\usepackage{epsfig}
\usepackage{pst-grad} 
\usepackage{pst-plot} 

\newcommand{\R}{\mathbb{R}}

\numberwithin{equation}{section}

\usepackage[notref,notcite]{showkeys}

\usepackage{amssymb}

\newcommand{\beq}{\begin{equation} }
\newcommand{\eqq}{\end{equation} }
\newcommand{\cuad}{{\sqcap\kern-.68em\sqcup}}

\newcommand{\equ}[1]{(\ref{#1})}

\def\beeq{\begin{equation}}
\def\eeq{\end{equation}}
\newcommand{\begeqaet}{\begin{eqnarray*}}
\newcommand{\eneqaet}{\end{eqnarray*}}

\let\Section=\section
\def\section{\setcounter{equation}{0}\Section}
\newtheorem{Lem}{Lemma}[section]
\newtheorem{Thm}{Theorem}[section]

\newtheorem{Prop}{Proposition}[section]
\newtheorem{Remark}{Remark}[section]

\begin{document}
\begin{center}{\bf\Large A critical nonlinear elliptic equation with non local regional diffusion }\medskip

\bigskip

\bigskip

{C\'esar E. Torres Ledesma}

Departamento de Matem\'aticas, \\
Universidad Nacional de Trujillo,\\
Av. Juan Pablo II s/n. Trujillo-Per\'u\\
{\sl  ctl\_576@yahoo.es}


\medskip

\medskip
\medskip
\medskip
\medskip

\end{center}

\centerline{\bf Abstract}

\medskip

In this article we are interested in  the nonlocal regional Schr\"odinger equation with critical exponent  
\begin{eqnarray*}
&\epsilon^{2\alpha} (-\Delta)_{\rho}^{\alpha}u + u = \lambda u^q + u^{2_{\alpha}^{*}-1} \mbox{ in } \mathbb{R}^{N},   \\
&
u \in H^{\alpha}(\mathbb{R}^{N}),
\end{eqnarray*}
where $\epsilon$ is a small positive parameter, $\alpha \in (0,1)$, $q\in (1,2_{\alpha}^{*}-1)$, $2_{\alpha}^{*} = \frac{2N}{N-2\alpha}$ is the critical Sobolev exponent, $\lambda >0$ is a parameter and $(-\Delta)_{\rho}^{\alpha}$ is a variational version of the regional laplacian, whose range of scope is a ball with radius $\rho(x)>0$. We study the existence of a ground state and we analyze the behavior of semi-classical solutions as $\varepsilon\to 0$. 

\medskip

\date{}

\setcounter{equation}{0}
\section{ Introduction}

In the present paper, we consider the existence and concentration phenomena of solutions to the nonlinear Schr\"odinger equation with non local regional diffusion  
\begin{eqnarray}\label{I01}
&\epsilon^{2\alpha} (-\Delta)_{\rho}^{\alpha}u + u = \lambda u^q + u^{2_{\alpha}^{*}-1} \mbox{ in } \mathbb{R}^{N},   \\
&
u \in H^{\alpha}(\mathbb{R}^{N}),\nonumber
\end{eqnarray}
where $\epsilon$ is a small positive parameter, $\alpha \in (0,1)$, $q\in (1,2_{\alpha}^{*}-1)$, $2_{\alpha}^{*} = \frac{2N}{N-2\alpha}$ is the critical Sobolev exponent, $\lambda >0$ is a parameter and  the operator $(-\Delta)_{\rho}^{\alpha}$ is a variational version of the non-local regional laplacian, with range of scope determined by the positive function $\rho \in C(\mathbb{R}^n, \mathbb{R}^+)$, which is defined as
$$
\int_{\mathbb{R}^n}(-\Delta)_{\rho}^{\alpha}u(x)\varphi (x)dx = \int_{\mathbb{R}^n}\int_{B(0, \rho (x))} \frac{[u(x+z) - u(z)][\varphi (x+z) - \varphi (x)]}{|z|^{n+2\alpha}}dxdx
$$ 

Now we make precise assumptions on the scope function $\rho$  we assume $\rho \in C(\mathbb{R}^{n},\mathbb{R}^{+})$ and it satisfies the following hypotheses:
\begin{itemize}
\item[$(\rho_{1})$] There are numbers $0<\rho_0 < \rho_\infty \leq \infty$ such that
$$\rho_{0 } \leq \rho (x) < \rho_\infty, \quad \forall x \in \mathbb{R}^{n} \quad\mbox{and}\quad
\lim_{|x| \to \infty} \rho (x) =\rho_\infty.
$$
\item[$(\rho_{2})$] For any $x_{0}\in \mathbb{R}^{n}$, the equation  
$$
|x| = \rho (x+x_{0}),\quad x\in\R^n,
$$
defines an $(n-1)$-dimensional surface of class $C^1$ in $\mathbb{R}^{n}$.
\item[$(\rho_{3})$] In case $\rho_\infty=\infty$ we further assume that there exists $a\in (0,1)$ such that
$$
\limsup_{|x| \to \infty} \frac{\rho (x)}{|x|}\le a.
$$

\end{itemize}

\section{Preliminaries}

For any $\alpha \in (0,1)$, the fractional Sobolev space $H^\alpha (\mathbb{R}^N)$ is defined by
$$
H^{\alpha}(\mathbb{R}^N) = \left\{ u\in L^2(\mathbb{R}^N):\;\; \frac{|u(x) - u(z)|}{|x-z|^{\frac{N+2\alpha}{2}}}  \in L^2(\mathbb{R}^N\times \mathbb{R}^N)\right\}
$$
endowed with the norm
$$
\|u\|_{\alpha} = \left(\int_{\mathbb{R}^N}|u(x)|^2dx + \int_{\mathbb{R}^N}\int_{\mathbb{R}^N}\frac{|u(x)-u(z)|^2}{|x-z|^{N+2\alpha}}dz dx \right)^{1/2}.
$$
For the reader's convenience, we review the main embedding result for $H^\alpha (\mathbb{R}^N)$.
\begin{Lem}\label{FSE}
\cite{EDNGPEV} Let $\alpha \in (0,1)$ such that $2\alpha < N$. Then there exist a constant $\mathfrak{C} = \mathfrak{C}(N, \alpha)>0$, such that
$$
\|u\|_{L^{2_{\alpha}^{*}}} \leq \mathfrak{C}\|u\|_{\alpha}
$$
for every $u\in H^{\alpha}(\mathbb{R}^N)$. Moreover, the embedding $H^\alpha (\mathbb{R}^N) \subset L^p(\mathbb{R}^N)$ is continuous for any $p\in [2, 2_{\alpha}^{*}]$ and is locally compact whenever $p\in [2, 2_{\alpha}^{*})$. 
\end{Lem}
Furthermore, we introduce the homogeneous fractional Sobolev space 
$$
\begin{aligned}
H_{0}^{\alpha}(\mathbb{R}^N) &= \left\{u\in L^{2_{\alpha}^{*}}(\mathbb{R}^N):\;\; |\xi|^{\alpha} \widehat{u} \in L^2(\mathbb{R}^N)  \right\}\\
&= \overline{C_{0}^{\infty}(\mathbb{R}^N)}^{\|.\|_{0}},
\end{aligned}
$$
where 
$$
\|u\|_{0}^2 = \int_{\mathbb{R}^N} |\xi|^{2\alpha}|\widehat{u}(\xi)|^2d\xi.
$$
Now we consider the best Sobolev constant $S$ as follows:
\begin{equation}\label{P01}
S = \inf_{u\in H_{0}^{\alpha}(\mathbb{R}^N)\setminus \{0\}} \frac{\int_{\mathbb{R}^N}\int_{\mathbb{R}^N}\frac{|u(x) - u(z)|^2}{|x-z|^{N+2\alpha}}dz dx}{\left( \int_{\mathbb{R}^N}|u(x)|^{2_{\alpha}^{*}}dx\right)^{2/2_{\alpha}^{*}}}.
\end{equation}
According to \cite{ACNT}, $S$ is attained by the function $u_0(x)$ given by
\begin{equation}\label{P02}
u_0(x) = \frac{c}{(\theta^2 + |x-x_0|^2)^{\frac{N-2\alpha}{2}}}, \;\;x\in \mathbb{R}^N,
\end{equation} 
where $c\in \mathbb{R}\setminus \{0\}$, $\theta >0$ and $x_0\in \mathbb{R}^N$ are fixed constants. For any $\epsilon >0$ and $x\in \mathbb{R}^N$, let
$$
U_\epsilon (x) = \epsilon^{-\frac{N-2\alpha}{2}}\tilde{u}\left( \frac{x}{\epsilon S^{1/2\alpha}} \right), \;\;\tilde{u}(x) = \frac{u_0(x)}{\|u_0\|_{L^{2_{\alpha}^{*}}}}
$$  
which is solution of the problem 
$$
(-\Delta)^{\alpha}u = |u|^{2_{\alpha}^{*} - 2}u,\;\;x\in \mathbb{R}^N.
$$

Given a function $\rho$ as above, we define
\begin{equation}\label{norma2}
\|u\|^2_\rho=\int_{\mathbb{R}^{n}}\int_{B(0,\rho (x))}\frac{|u(x) - u(z)|^{2}}{|x-z|^{n+2\alpha}}dzdx + \int_{\mathbb{R}^{n}}u(x)^{2}dx
\end{equation}
and the space
$$
H_{\rho}^{\alpha}(\mathbb{R}^{n})=\{u\in L^2(\mathbb{R}^{n})\,/\, \|u\|^2_\rho<\infty\}.
$$
This space is very natural for the study of our problem. Furthermore, we have the following result
\begin{Prop}\label{PFCTprop}
\cite{PFCT1} If $\rho$ satisfies ($H_{1}$) there exists a constant $C = C(n, \alpha, \rho_{0})> 0$ such that
$$
\|u\|_{\alpha} \leq C \|u\|_{\rho}.
$$
\end{Prop}

\begin{Remark}\label{FSnta1}
 By Proposition \ref{PFCTprop} we have that $H_{\rho}^{\alpha}(\mathbb{R}^{n}) \hookrightarrow H^{\alpha}(\mathbb{R}^{n})$ is continuous and  then, by  Theorem \ref{FSE}, we have that
    $H_{\rho}^{\alpha}(\mathbb{R}^{n}) \hookrightarrow L^{q}(\mathbb{R}^{n})$ is continuous for any $q\in [2,2_{\alpha}^{*}]$, and there exists $\mathfrak{C}_{q}>0$ such that 
$$
\|u\|_{L^q} \leq \mathfrak{C}_q\|u\|_{\rho},\;\;\mbox{for every}\;\;u\in H_{\rho}^{\alpha}(\R^n)\;\;\mbox{and}\;\;q\in [2,2_{\alpha}^{*}].
$$ 
Furthermore $H_{\rho}^{\alpha}(\mathbb{R}^{n}) \hookrightarrow L_{loc}^{q}(\mathbb{R}^{n})$ is compact for any $q\in [2,2_{\alpha}^{*})$.
  \end{Remark}
    \begin{Remark}\label{FSnta1x}
    Since $\|u\|_{\rho} \leq \|u\|_\alpha$, under the condition $(\rho_{1})$  Proposition \ref{PFCTprop} implies  $\|\cdot\|_{\rho}$
    and  $\|\cdot\|$ are equivalent norms in  $H^{\alpha} (\mathbb{R}^{n})$.
\end{Remark}

\begin{Lem}\label{CClem}
\cite{PFCT1} Let $n\geq 2$. Assume that $\{u_n\}_{n\in \mathbb{N}}$ is bounded in $H_{\rho}^{\alpha}(\mathbb{R}^N)$ and it satisfies 
$$
\lim_{n\to \infty}\sup_{y\in \mathbb{R}^N} \int_{B(y, R)} |u_n(x)|^2dx = 0,
$$
where $R>0$. Then $u_n \to 0$ in $L^q(\mathbb{R}^N)$ for $q\in (2, 2_{\alpha}^{*})$.
\end{Lem}

Now, we consider the limit equations. namely
\begin{eqnarray}\label{L01}
&(-\Delta)^{\alpha}u + u = \lambda u^q + u^{2_{\alpha}^{*} -1}  \mbox{ in } \mathbb{R}^{N},   \\
&u \in H^{\alpha}(\mathbb{R}^{N}).\nonumber
\end{eqnarray}
This equation was study by Shang, Zhang and Yang in \cite{XSJZYY}. The solution of problem (\ref{L01}) are the critical point of the functional 
$$
I(u) = \frac{1}{2}\|u\|_{\alpha}^{2} - - \frac{\lambda}{q+1}\int_{\mathbb{R}^n}u_+^{q+1}dx - \frac{1}{2_\alpha^{*}}\int_{\mathbb{R}^n} u_+^{2_{\alpha}^{*}}dx.
$$    
Furthermore, they studied the existence of ground state solutions to (\ref{L01}), namely, function in $\mathcal{N}=\{u\in H^{\alpha}(\mathbb{R}^n)\setminus \{0\}: I'(u)u=0\}$ such that
$$
C^* = \inf_{u\in \mathcal{N}} I(u)
$$ 
is achieved and they got the following characterization: 
$$C^* = \inf_{u\in H^\alpha(\mathbb{R}^n)\setminus \{0\}}\sup_{t\geq 0}I(tu) = C,$$
where $C = \inf_{\gamma \in \Gamma}\sup_{t\in [0,1]}I(\gamma (t))>0$ is the mountain pass critical value. 

On the other hand, assuming that if $q\in (1, 2_{\alpha}^{*} - 1)$, they have proved that there exists $\lambda_0>0$ such that for all $\lambda \geq \lambda_0$, the critical value satisfies 
\begin{equation}\label{L02}
0< C < \frac{\alpha}{n} S^{n/2\alpha},
\end{equation}  
where $S$ is the best Sobolev constant given by (\ref{P01}) and problem (\ref{L01}) has a nontrivial ground state solution.

\section{Ground state}

Let $\epsilon =1$ and consider the following problem
\begin{eqnarray}\label{G01}
&(-\Delta)_{\rho}^{\alpha}u + u = \lambda u^q + u^{2_{\alpha}^{*} -1}  \mbox{ in } \mathbb{R}^{N},   \\
&u \in H^{\alpha}(\mathbb{R}^{N}).\nonumber
\end{eqnarray}

We recall that $u\in H_{\rho}^{\alpha}(\mathbb{R}^n) \setminus \{0\}$ is a solution of (\ref{I01}) if $u(x)\geq 0$ and
$$
\begin{aligned}
\langle u, \varphi \rangle_\rho= \lambda \int_{\mathbb{R}^n}u_+^q\varphi dx + \int_{\mathbb{R}^n} u_+^{2_{\alpha}^{*} - 1}\varphi dx, \;\;\forall \varphi \in H_{\rho}^{\alpha}(\mathbb{R}^n).
\end{aligned}
$$
where
$$
\langle u, \varphi\rangle_{\rho}= \int_{\mathbb{R}^n}\int_{B(0, \rho (x))} \frac{[u(x+z) - u(x)][\varphi (x+z)-\varphi (x)]}{|z|^{n+2\alpha}}dz dx + \int_{\mathbb{R}^n}u\varphi dx
$$
and $u_+= \max\{u, 0\}$.

In order to find solution for problem (\ref{G01}), we consider the functional $I_{\rho}: H_{\rho}^{\alpha}(\mathbb{R}^n) \to \mathbb{R}$ defined as
$$
I_{\rho}(u) = \frac{1}{2}\|u\|_{\rho}^2 - \frac{\lambda}{q+1}\int_{\mathbb{R}^n}u_+^{q+1}dx - \frac{1}{2_\alpha^{*}}\int_{\mathbb{R}^n} u_+^{2_{\alpha}^{*}}dx,
$$
which is well defined and belongs to $C^1(H_{\rho}^{\alpha}(\mathbb{R}^n), \mathbb{R}^n)$ with Fr\'echet derivative
$$
I'_{\rho}(u)v = \langle u, v\rangle_{\rho} - \lambda \int_{\mathbb{R}^n}u_+^qvdx - \int_{\mathbb{R}^n}u_{+}^{2_{\alpha}^{*}-1}vdx. 
$$. 

Now, we start recalling that the functional $I_{\rho}$ satisfies the mountain pass geometry conditions 

\begin{Lem}\label{Glm1}
The functional $I_{\rho_\epsilon}$ satisfies the following conditions:
\begin{enumerate}
\item There exist $\beta, \delta>0$, such that $I_{\rho}(u) \geq \beta$ if $\|u\|_{\rho} = \delta$.
\item There exists an $e\in H_{\rho}^{\alpha}(\mathbb{R}^n)$ with $\|e\|_{\rho}$ such that $I_{\rho}(e) <0$.
\end{enumerate}
\end{Lem}

From the previous Lemma, by using the mountain pass theorem without $(PS)_c$ condition (\cite{MW}) it follows that there exists a $(PS)_{C_{\rho}}$ sequence $\{u_k\} \subset H_{\rho}^{\alpha}(\mathbb{R}^n)$ such that
\begin{equation}\label{G02}
I_{\rho}(u_k) \to C_{\rho}\quad \mbox{and}\quad I'_{\rho}(u_k) \to 0,
\end{equation}
where 
\begin{equation}\label{G03}
C_{\rho} = \inf_{\gamma \in \Gamma_{\rho}} \sup_{t\in [0,1]} I_{\rho}(\gamma (t)) >0,
\end{equation}
and $\Gamma_{\rho} = \{\gamma \in C([0,1], H_{\rho}^{\alpha}(\mathbb{R}^n)):\;\;\gamma (0) = 0, I_{\rho}(\gamma (1))<0\}$. 

Also, we define 
$$
C^* = \inf_{u\in \mathcal{N}_{\rho}} I_{\rho}(u)
$$
where
$$
\mathcal{N}_{\rho} = \{u\in H_{\rho}^{\alpha}(\mathbb{R}^n)\setminus \{0\}:\;\;I'_{\rho}(u)u=0\}.
$$

\begin{Lem}\label{Glm2}
$C^*>0$
\end{Lem}

\begin{proof}
Suppose by contradiction that $C^* = 0$. Then there exists $u_n \in \mathcal{N}_\rho$. Then there exists $u_k \in \mathcal{N}_\rho$ such that
$$
I_\rho (u_k) \to C^* = 0,\;\;\mbox{as}\;\;k \to \infty.
$$
But, since $u_k \in \mathcal{N}_\rho$ we have
$$
\begin{aligned}
I_\rho (u_k) &= I_\rho (u_k) - \frac{1}{2}I'_\rho (u_k)u_k\\
&= \lambda \left( \frac{1}{2}-\frac{1}{q+1} \right)\|u_{k+}\|_{L^{q+1}}^{q+1} + \frac{\alpha}{n} \|u_{k+}\|_{L^{2_{\alpha}^{*}}}^{2_{\alpha}^{*}} \to 0, \;\;\mbox{as}\;\;k \to \infty.
\end{aligned}
$$
So
$$
\|u_{k+}\|_{L^{2_{\alpha}^{*}}} \to 0\;\;\mbox{and}\;\;\|u_{k+}\|_{L^{q+1}} \to 0,\;\;\mbox{as}\;\;k \to \infty.
$$
Therefore
\begin{equation}\label{G04}
\|u_k\|_{\rho} \to 0,\;\;\mbox{as}\;\;k\to \infty.
\end{equation}
On the other hand, since $0\neq u_k \in \mathcal{N}_\rho$, then by Remark \ref{FSnta1} we have
$$
1\leq \lambda C_{q+1}\|u_k\|_{\rho}^{q-1} + C_{2_{\alpha}^{*}}\|u_k\|^{2_{\alpha}^{*} - 2},\;\;\forall k.
$$ 
Combining this inequality with (\ref{G04}) we get a contradiction. This proves the Lemma. 
\end{proof}

\begin{Lem}\label{Glm3}
Let $C_\rho$ given by (\ref{G02}) and (\ref{G03}). Then 
$$C^* = \inf_{u\in H_{\rho}^{\alpha}(\mathbb{R}^n)\setminus \{0\}} \max_{t\geq 0} I_{\rho}(tu) = C_\rho$$.
\end{Lem}

\begin{proof}
We note that our nonlinearity $f(t) = \lambda t^{q} + t^{2_{\alpha}^{*}-1}$, $t>0$, is a $C^1$ function and $\frac{f(t)}{t}$ is a strictly increasing function. Let $u\in \mathcal{N}_{\rho}$, then we can show that the function $h(t) = I_{\rho}(tu)$, $t\neq 0$ has a unique maximum point $t_u$ such that 
\begin{equation}\label{G05}
I_{\rho}(t_uu) = \max_{t\geq 0}I_{\rho}(tu).
\end{equation}
Furthermore, we can show that $t_u =1$. Now choose $t_0\in \mathbb{R}$ and $\tilde{u} = t_0 u$ such that $I_{\rho}(\tilde{u}) <0$. Then $\gamma (t) = t \tilde{u} \in \Gamma_{\rho}$ then $I_\rho (u) \geq C_\rho$, that is,
\begin{equation}\label{G06}
C^* \geq C_{\rho}.
\end{equation}
On the other hand, let $\{u_k\}$ be the $(PS)_{C_\rho}$ sequence satisfying (\ref{G02}) and (\ref{G03}). Since $\{u_k\}$ is bounded, then $I'_{\rho}(u_k)u_k \to 0$ as $k\to \infty$, moreover, from (\ref{G05}) for each $k$, there is a unique $t_k\in \mathbb{R}^+$ such that 
\begin{equation}\label{G07}
I'_{\rho}(t_ku_k)t_ku_k = 0,\;\;\forall k.
\end{equation}     
Hence $t_ku_k \in \mathcal{N}_\rho$.

Now we note that by (\ref{G07}), we have
\begin{equation}\label{G08}
\|u_k\|_{\rho}^{2} = \lambda t_{k}^{q-1}\|u_{k+}\|_{L^{q+1}}^{q+1} + t_{k}^{2_{\alpha}^{*}-2}\|u_{k+}\|_{L^{2_{\alpha}^{*}}}^{2_{\alpha}^{*}}, \;\;\forall k.
\end{equation}
So $t_k$ does not converge to $0$; otherwise, since $\{u_k\}$ is bounded in $H_{\rho}^{\alpha}(\mathbb{R}^n)$, using (\ref{G08}) we obtain
$$
\|u_k\|_{\rho} \to 0,\;\;\mbox{as}\;\;k\to \infty,
$$ 
which is imposible since $C_\rho >0$. Also, $t_k$ does not go to infinity. In fact, by (\ref{G08}) we get
\begin{equation}\label{G09}
\frac{\|u_k\|_{\rho}^2}{t_{k}^{2_{\alpha}^{*} - 2}} = \lambda t_{k}^{q-1-2_{\alpha}^{*}}\|u_{k+}\|_{L^{q+1}}^{q+1} + \|u_{k+}\|_{L^{2_{\alpha}^{*}}}^{2_{\alpha}^{*}}, \;\;\forall k.
\end{equation}
So, assuming that $t_k \to \infty$, as $k \to \infty$, by (\ref{G09}) we get that 
$$
u_k \to 0 \;\;\mbox{in}\;\;L^{2_{\alpha}^{*}}(\mathbb{R}^n),\;\;\mbox{as}\;\;k \to \infty.
$$ 
Then using interpolation inequality it follows that
\begin{equation}\label{G10}
\|u_{k+}\|_{L^{q+1}}^{q+1}\to 0\;\;\mbox{as}\;\;k\to \infty.
\end{equation}
Moreover, since $I'_{\rho}(u_k)u_k \to 0$, as $k \to \infty$, we obtain,
\begin{equation}\label{G11}
\|u_k\|_{\rho}^{2} = \lambda \|u_{k+}\|_{L^{q+1}}^{q+1} + \|u_k\|_{L^{2_{\alpha}^{*}}}^{2_{\alpha}^{*}} + o(1),\;\;\mbox{as}\;\;k \to \infty.
\end{equation} 
So, by (\ref{G10}) and (\ref{G11}), we conclude that $\|u_k\|_{\rho}^{2} \to 0$, as $k \to \infty$, contradicting $C_{\rho}>0$. Hence, the sequence $\{t_k\}$ is bounded and there exists $t_0\in (0, \infty)$ such that (up to subsequence) $t_k \to t_0$ as $k \to \infty$. 

Now, from (\ref{G08}) and (\ref{G11}) we obtain 
\begin{equation}\label{G12}
o(1) = \lambda (t_{k}^{q-1} - 1)\|u_k\|_{L^{q+1}}^{q+1} + (t_{k}^{2_{\alpha}^{*} - 2} 1)\|u_k\|_{L^{2_{\alpha}^{*}}}^{2_{\alpha}^{*}},\;\;\mbox{as}\;\;k \to \infty.
\end{equation}
From where $t_0 = 1$, namely
\begin{equation}\label{G13}
t_k \to 1,\;\;\mbox{as}\;\;k \to \infty.
\end{equation} 

Therefore, by (\ref{G13}) and recalling that $t_k u_k \in \mathcal{N}_\rho$, we get
$$
\begin{aligned}
C^{*} &\leq I_\rho (t_k u_k)\\
&= t_{k}^{2}\left[ I_\rho (u_k) + \frac{\lambda}{q+1} (1-t_{k}^{q-1})\|u_{k+}\|_{L^{q+1}}^{q+1} + \frac{1}{2_{\alpha}^{*}} (1 - t_{k}^{2_{\alpha}^{*} - 2})\|u_{k+}\|_{L^{2_{\alpha}^{*}}}^{2_{\alpha}^{*}} \right]\\
& = t_{k}^2I_{\rho}(u_k) + o(1)\\
& = (t_k^2 - 1)I_{\rho}(u_k) + I_{\rho (u_k)} + o(1).
\end{aligned}
$$
Passing to the limit we obtain $C^* \leq C_\rho$.

On the other hand, By the previous comments, for any $u\in H_{\rho}^{\alpha}(\mathbb{R}^{n})\setminus \{0\}$ there is a unique $t_{u} = t(u) >0$ such that
$
t_{u}u \in \mathcal{N}_{\rho},
$
then
$$
C^{*} \leq \inf_{u\in H_{\rho}^{\alpha}(\mathbb{R}^{n})\setminus \{0\}} \max_{t\geq 0} I_{\rho}(tu).
$$
Moreover, for any $u\in \mathcal{N}_{\rho}$, we have
\begin{eqnarray*}
I_{\rho}(u) & = & \max_{t\geq 0} I_{\rho}(tu) \geq \inf_{u\in H_{\rho}^{\alpha}(\mathbb{R}^{n})\setminus \{0\}} \max_{t\geq 0} I_{\rho}(tu)
\end{eqnarray*}
so
$$
C^{*}=\inf_{\mathcal{N}_{\rho}} I_{\rho}(u) \geq \inf_{u\in H_{\rho}^{\alpha}(\mathbb{R}^{n})\setminus \{0\}} \max_{t\geq 0} I_{\rho}(tu).
$$
\end{proof}

\begin{Remark}\label{Gnta1}
Suppose that $(\rho_1)$ holds and without loss of generality take $\epsilon = 1$, then  
$$
C_{\rho} < C.
$$ 
In fact, let $u$ be a critical point of $I$ with critical value $C$ and for any $y\in \mathbb{R}^n$, define $u_y(x) = u(x+y)$. Then  for any $t>0$ we have
$$
C = I(u_y) \geq I(tu_y) > I_{\rho}(tu_y).
$$
By Lemma \ref{Glm2}, we can take $t^*>0$ such that $t^*u_y \in \mathcal{N}_\rho$ and 
$$
I_\rho (t^* u_y) = \sup_{t>0}I_\rho (tu_y),
$$
consequently $C> I_\rho (t^*u_y) \geq C_\rho$. In the same way we can show that
$$
C_\rho < C_{\rho_\infty} < C.
$$
\end{Remark}
\begin{Remark}\label{Gnta2}
According to (\ref{L02}) and by Remark \ref{Gnta1} we have
$$
0<C_\rho< \frac{\alpha}{n}S^{n/2\alpha}.
$$
\end{Remark}

\begin{Lem}\label{Glm4}
Suppose that $C_\rho < C$. Then there are $\nu, R>0$ such that 
$$
\int_{B(0, R)} u_{k}^{2}(x)dx \geq \nu,\;\;\mbox{for all}\;\;k\in \mathbb{N}. 
$$
\end{Lem}

\begin{proof}
By Lemma \ref{Glm3}, for each $k\in \mathcal{N}$, there exist $t_k \subset \mathbb{R}$ such that $t_k \to 1$ and 
$$
I_\rho (t_k u_k) = \max_{t\geq 0} I_{\rho} (tu_k).
$$ 
Furthermore
$$
I_{\rho}(u_k) = I_{\rho}(t_ku_k) + o(1) \geq I_{\rho}(tu_k)+o(1),\;\;\mbox{for all}\;\;t>0.
$$

Now we consider two cases, namely, when $\rho_\infty = \infty$ and $\rho_\infty < \infty$. In the first case, by $(\rho_1)$, for every $\epsilon >0$ there exist $R>0$ such that 
$$
\mathbb{R}^n \setminus B(0, \rho (x)) \subset \mathbb{R}^n \setminus B(0, \frac{1}{\epsilon})\;\;\mbox{whenever}\;\;|x|>R.
$$
Then
\begin{equation}\label{G14}
\begin{aligned}
I_\rho (tu_k) &= I(tu_k) - \frac{t^2}{2}\int_{\mathbb{R}^n}\int_{\mathbb{R}^n\setminus B(0, \rho (x))}\frac{|u_k(x+z) - u_k(x)|^2}{|z|^{n+2\alpha}}dz dx\\
&\geq I(tu_k) -  \frac{|S^{n-1}|t^2}{\alpha \rho_0^{2\alpha}}\|u_k\|_{L^2(B(0, R))}^{2} + \frac{|S^{n-1}|t^2\epsilon^{2\alpha}}{\alpha}\|u_k\|_{L^2(\mathbb{R}^n)}^{2},\;\;\mbox{for all}\;\;t>0.
\end{aligned}
\end{equation}
 If $\{\tau_k\}$ is the bounded real sequence seven by
 $$
 I(\tau_k u_k) = \max_{t\geq 0}I(tu_k),
 $$
 we obtain
 \begin{equation}\label{G15}
I_\rho (t_ku_k) \geq I(\tau_{k}u_k) -  \frac{|S^{n-1}|\tau_{k}^2}{\alpha \rho_0^{2\alpha}}\|u_k\|_{L^2(B(0, R))}^{2} + \frac{|S^{n-1}|\tau_{k}^2\epsilon^{2\alpha}}{\alpha}\|u_k\|_{L^2(\mathbb{R}^n)}^{2}.
 \end{equation}
 If 
 \begin{equation}\label{G16} 
 \int_{B(0, R)} u_k^2(x)dx \to 0\;\;\mbox{as}\;\;k\to \infty,
 \end{equation}
 from (\ref{G14}), (\ref{G15})  and (\ref{G16}) yield
 $$
 C_\rho \geq C,
 $$
 which is a contradiction with Remark \ref{Gnta1}. 
 
Now we analyze the case  $\rho_{\infty} < +\infty$. In this case we compare the functionals $I_{\rho}$ and 
$I_{\rho_{\infty}}$ writing
\begin{equation}\label{SEPeq17n}
I_{\rho}(u) = I_{\rho_{\infty}}(u) - \frac{1}{2}\int_{\mathbb{R}^{n}}\int_{B(0, \rho_{\infty})\setminus B(0,\rho (x))} \frac{|u(x+z) - u(x)|^{2}}{|z|^{n+2\alpha}}dzdx.
\end{equation}
By hypothesis 
($\rho_{1}$), for any $\epsilon >0$ there is $R>0$ such that
$$
0< \rho_{\infty} - \rho(x)  < \epsilon,\;\mbox{whenever}\;|x|>R. 
$$ 
Then, we obtain
\begin{eqnarray*}
I_{\rho}(t u_{k})  \ge  I_{\rho_{\infty}}(t u_{k}) - C(\epsilon)t^2\|u_k\|_{L^{2}}^{2} - C t^2\|u_k\|_{L^{2}(B(0,R))}^{2}.
\end{eqnarray*}
where 
\begin{eqnarray*}
C(\epsilon)=\frac{2|S^{n-1}|}{\alpha} \left( \frac{1}{(\rho_{\infty} - \epsilon)^{2\alpha}}  - \frac{1}{\rho_{\infty}^{2\alpha}} \right)
\,\,\,\mbox{and}\,\,\,
C= \frac{2|S^{n-1}|}{\alpha}\left( \frac{1}{\rho_{0}^{2\alpha}}  - \frac{1}{\rho_{\infty}^{2\alpha}}\right).
\end{eqnarray*}
Proceeding as before, by (\ref{G16}) we get $C_\rho \geq C_{\rho_\infty}$,  which is a contradiction with Remark \ref{Gnta1}. 
\end{proof}

The next result shows the existence of positive solution to (\ref{I01}) with $\epsilon = 1$.

\begin{Thm}\label{Gtm1}
Suppose that $\lambda>0$, $q>1$ and $(\rho_1)$ hold. Then, problem (\ref{I01}) with $\epsilon=1$ possesses a positive ground state solution. 
\end{Thm}

\begin{proof}
Using (\ref{G02}), Lemma \ref{Glm4} and the Sobolev embedding we have
$$
\int_{B(0, R)} u^2(x)dx \geq \nu >0,
$$
which proves that $u\neq 0$. Furthermore, by standard arguments we have
$$
I'_\rho (u)\varphi = 0\;\;\mbox{for all}\;\;\varphi \in H_{\rho}^{\alpha}(\mathbb{R}^n),
$$ 
so choosing $\varphi = u_{-}(x) \max \{-u(x), 0\}$ and noting that for $x,z\in\R^n$ we have
\begin{eqnarray*}
(u(x+z)-u(x))(u_-(x+z)-u_-(x))&=&-u_+(x+z)u_-(x)-u_+(x)u_-(x+z)\\
&&-(u_-(x+z)-u_-(x))^2\le 0, 
\end{eqnarray*}
we can conclude that $\|u_{-}\|_{\rho} = 0$, thus $u(x)\geq 0$ a.e.$x\in \mathbb{R}^n$.

Moreover, from Lemma \ref{Glm3}
\begin{equation}\label{G18}
C_\rho \leq \max_{t\geq 0} I_\rho (tu) = I_{\rho}(u).
\end{equation}
On the other hand, we have
$$
\begin{aligned}
C_\rho &= I_{\rho}(u_k) - \frac{1}{2}I'_\rho (u_k)u_k + o_k(1)\\
&= \lambda \left( \frac{1}{2} - \frac{1}{q+1} \right)\int_{\mathbb{R}^n}u_{k+}^{q+1}dx + \frac{1}{n}\int_{\mathbb{R}^n}u_{k+}^{2_{\alpha}^{*}}dx + o_k(1).
\end{aligned}
$$ 
Applying Fatou's Lemma to last inequality, we obtain
\begin{equation}\label{G19}
C_\rho \geq \lambda \left(\frac{1}{2} - \frac{1}{q+1}  \right)\int_{\mathbb{R}^n}u^{q+1}dx +\frac{1}{n}\int_{\mathbb{R}^n}u^{2_{\alpha}^{*}}dx = I_\rho (u) - \frac{1}{2}I'_\rho (u)u = I_\rho (u)
\end{equation}    
From (\ref{G18}) and (\ref{G19}) we obtain
$$
I_\rho (u) = C_\rho,
$$
and hence $u$ is a least energy solution and the proof is finished.
\end{proof}

\noindent
{\bf Proof of Theorem \ref{}.} In what follows, we denote by $\{u_k\} \subset H_{\rho_\epsilon}^{\alpha}(\mathbb{R}^N)$ a sequence satisfying
$$
I_{\rho_{\epsilon}}(u_k) \to C_{\rho_{\epsilon}} \quad \mbox{and} \quad I'_{\rho_{\epsilon}}(u_k) \to 0.
$$	
If $u_k \rightharpoonup 0$ in $H_{\rho_\epsilon}^{\alpha}(\mathbb{R}^N)$, then 
\begin{equation}\label{lim0}
u_k \to 0\;\;\mbox{ in}\;\;L_{loc}^{p}(\mathbb{R}^n)\;\;\mbox{for}\;\; p\in [2, 2_{\alpha}^{*}).
\end{equation} 
By $(\rho_1)$, we obtain 
$$
\begin{aligned}
I_{\rho_\epsilon} (tu_k) & = I_{\frac{\rho_\infty}{\epsilon}} (tu_k) - \frac{t^2}{2}\int_{\mathbb{R}^n}\int_{B(0, \frac{\rho_\infty}{\epsilon})\setminus B(0, \frac{\rho (\epsilon x)}{\epsilon})} \frac{|u_k(x+z) - u_k(x)|^2}{|z|^{n+2\alpha}}dz dx
\end{aligned}
$$
where
$$
\begin{aligned}
I_{\frac{\rho_\infty}{\epsilon}} (u) &= \frac{1}{2}\left( \int_{\mathbb{R}^n}\int_{B(0, \frac{\rho_\infty}{\epsilon})} \frac{|u(x+z) - u(x)|^2}{|z|^{n+2\alpha}}dxdx + \int_{\mathbb{R}^n}u^2dx \right) \\&- \frac{\lambda}{q+1}\int_{\mathbb{R}^n} u_+^{q+1}dx - \frac{1}{2_{\alpha}^{*}}\int_{\mathbb{R}^n} u_{+}^{2_{\alpha}^{*}}dx.
\end{aligned}
$$
Now we know that there exists a bounded sequence $\{\tau_k\}$ such that  
$$
I_{\frac{\rho_\infty}{\epsilon}}(\tau_k u_k) \geq C(\frac{\rho_\infty}{\epsilon}), 
$$
where
$$
C(\frac{\rho_\infty}{\epsilon}) = \inf_{v\in H^{\alpha}(\mathbb{R})\setminus \{0\}} \sup_{t\geq 0}I_{\frac{\rho_\infty}{\epsilon}}(tv)
$$
Thus, 
$$
\begin{aligned}
C_{\rho_{\epsilon}} &\geq C(\frac{\rho_\infty}{\epsilon})- \frac{|S^{n-1}|}{\alpha}\left(\frac{1}{\rho_0^{2\alpha}} - \frac{1}{\rho_{\infty}^{2\alpha}} \right)\tau_k^2  \epsilon^{2\alpha}\|u_k\|_{L^2(B(0, \frac{R}{\epsilon}))}^2\\
& - \frac{|S^{n-1}|}{\alpha}\left( \frac{1}{(\rho_\infty - \delta)^{2\alpha}} - \frac{1}{\rho_{\infty}^{2\alpha}} \right)\tau_k^{2}\epsilon^{2\alpha}\|u_k\|_{L^2(\mathbb{R}^n)}^{2}
\end{aligned}
$$
Taking the limit as $k\to \infty$, and after $\delta \to 0$, we find
\begin{equation}\label{eq21}
C_{\rho_\epsilon} \geq C(\frac{\rho_\infty}{\epsilon})
\end{equation}
A standard argument shows that
$$
\liminf_{\epsilon \to 0}C(\frac{\rho_\infty}{\epsilon}) \geq C.
$$ 
Therefore, if there is $\epsilon_k \to 0$ such that the $(PS)_{C_{\rho_{\epsilon_k}}}$ sequence has weak limit equal to zero, we must have
$$
C_{\rho_{\epsilon_k}} \geq C(\frac{\rho_\infty}{\epsilon_k}), \quad \forall k \in \mathbb{N},
$$
leading to 
$$
\liminf_{n \to +\infty} C_{\rho_{\epsilon_n}} \geq C,
$$
which contradicts Remark \ref{Gnta1}.  This proves that the weak limit is non trivial for $\epsilon >0$ small enough and standard arguments show that its energy is equal to $C_{\rho_{\epsilon}}$, showing the desired result. $\Box$

\section{Concentration Behaviour}

In this section we make a preliminary analysis of the asymptotic behavior  of the functional associated to equation \equ{I01} when $\epsilon\to 0$.  As is point up in \cite{PFCT1}, the scope function $\rho$, that describes the size of the ball of the influential region of the non-local operator, plays a key role in deciding the concentration point of ground states of the equation. Even though, at a first sight, the minimum point of $\rho$ seems to be the concentration point, there is a non-local effect that needs to be taken in account. We define the concentration function
\begin{equation}\label{C01}
\mathcal{H}(x) = -\frac{|S^{n-1}|}{2\alpha}\left( \frac{1}{\rho(x)^{2\alpha}} - \frac{1}{\rho_{\infty}^{2\alpha}}\right) + \frac{1}{2}\int_{\mathcal{C}^+(x)} \frac{dy}{|y|^{n+2\alpha}} - \frac{1}{2}\int_{\mathcal{C}^-} \frac{dy}{|y|^{n+2\alpha}},
\end{equation}
where the sets $\mathcal{C}^{+}({x})$ and $\mathcal{C}^{-}({x})$ are defined as follows
$$
\mathcal{C}^{-}({x}) = \{ y\in\R^n : \;\; \rho(x+ y ) < |y| < \rho ({x})\}
$$ 
and  
$$
\mathcal{C}^{+}({x}) = \{y\in\R^n  :\;\; \rho({x}) < |y| < \rho (x + y )\}.
$$
We start with  some basic properties of the function $\mathcal{H}$.  
\begin{Lem}\label{CClm01}
\cite{PFCT1} Assuming $\rho$ satisfies $(\rho_1)-(\rho_3)$, the function $\mathcal{H}$ is continuous and 
\begin{equation}\label{C02}
\lim_{|x|\to\infty}\mathcal{H}(x)=0.
\end{equation}
Moreover, there exists $x_{0}\in \mathbb{R}^{n}$ such that
\begin{equation}\label{C03}
\inf_{x\in\R^n}\mathcal{H}(x)=\mathcal{H}(x_{0}) < 0.
\end{equation}
\end{Lem}

Along this section we will consider a sequence of functions $\{w_{m}\} \subset H^{\alpha}(\mathbb{R}^{n})$ such that $\|w_{m} - w\|_{L^{2}(\mathbb{R}^{2})} \to 0$, where $w\in H^{\alpha}(\mathbb{R}^{n})$.  We will also consider  sequences $\{z_m\}\subset \R^n$ and $\{\epsilon_m\}\subset \R$ and assume that $\epsilon_m\to 0$ as $m\to\infty$. We define
 $\overline{\rho}_{m}$ as
\begin{equation}\label{C04}
\overline{\rho}_{m}(x) = \frac{1}{\epsilon_m} \rho (\epsilon_m x + \epsilon_m z_m),
\end{equation}
and the functional $I_{\bar{\rho}_m}$ defined as 
\begin{equation}\label{C05}
\begin{aligned}
I_{\bar{\rho}_m}(u) =& \frac{1}{2}\left( \int_{\mathbb{R}^n}\int_{B(0, \bar{\rho}_{m}(x))}\frac{|u(x+z) - u(x)|^2}{|z|^{n+2\alpha}}dz dx + \int_{\mathbb{R}^n}u^2dx \right) \\
&- \frac{\lambda}{q+1}\int_{\mathbb{R}^n}u_{+}^{q+1}dx - \frac{1}{2_{\alpha}^{*}}\int_{\mathbb{R}^n}u_{+}^{2_{\alpha}^{*}}dx.
\end{aligned}
\end{equation}
We will be considering the functionals
$$
I_{\frac{\rho_{\infty}}{\epsilon_{m}}}, I_{\frac{\rho(\bar x)}{\epsilon_{m}}}\quad\mbox{and the functional} \;\;I\;\;\mbox{in}\;\;\mathbb{R}^n\;\; (\mbox{with} \;\;\rho \equiv \infty).
$$
As in \cite{PFCT1} we have the following key Theorem
\begin{Thm}\label{CCtm01}
Under  hypotheses $(\rho_{1})-(\rho_{3})$, we assume as above that 
$w_{m}, w  \in H^{\alpha}(\mathbb{R}^{n})$ are such that $\|w_{m} - w\|_{L^{2}(\mathbb{R}^{2})} \to 0$ and  $\epsilon_m\to 0$, as $m\to\infty$. Then we have:

i)
 If $\epsilon_m z_m\to \bar x$ then 
\begin{equation}\label{C06}
\lim_{m \to \infty} \frac{I_{\overline{\rho}_{m}}(w_{m}) - I_{\frac{\rho_{\infty}}{\epsilon_{m}}}(w_{m})}{\epsilon_{m}^{2\alpha}} = \|w\|_{L^{2}}^{2} \mathcal{H}(\overline{x})\qquad \mbox{and}
\end{equation}

ii) If $|\epsilon_{m}|z_{{m}} \to\infty$ then 
\begin{equation}\label{C07}
\lim_{m \to \infty} \frac{I_{\overline{\rho}_{m}}(w_{m}) - I_{\frac{\rho_{\infty}}{\epsilon_{m}}}(w_{m})}{\epsilon_{m}^{2\alpha}} = 0.
\end{equation}
\end{Thm}

Now, we rescaling equation \equ{I01}, for this purpose we define  
$\rho_{\epsilon}(x) = \frac{1}{\epsilon} \rho (\epsilon x)$ and consider the rescaled equation
\begin{eqnarray}\label{C08}
(-\Delta)_{\rho_{\epsilon}}^{\alpha} v + v = \lambda v^{q} + u^{2_{\alpha}^{*} - 1}, &\mbox{in} \quad\mathbb{R}^{n}
\end{eqnarray}
and we see that  $u$ is a weak solution of 
 (\ref{I01}) if and only if  $v_{\epsilon}(x) = u(\epsilon x)$ is a weak solution of (\ref{C08}). 

In order to study equations (\ref{I01}) and (\ref{C08}),  we consider the functional $I_{\rho_{\epsilon}}$ on the   $\epsilon$-dependent Hilbert space  $H_{\rho_\epsilon}^{\alpha}(\mathbb{R}^{n})$ with inner product $\langle\cdot,\cdot\rangle_{\rho_\epsilon}$. 
The functional $I_{\rho_{\epsilon}}$ is of class $C^1$ in $H_{\rho_{\epsilon}}^{\alpha}(\mathbb{R}^{n})$ and 
 the critical points of $I_{\rho_{\epsilon}}$ are the  weak solutions of (\ref{C08}). We further introduce
$$
\mathcal{N}_{\rho_{\epsilon}} = \{v \in H_{\rho_{\epsilon}}^{\alpha}(\mathbb{R}^{n})\setminus \{0\}:\;\; I'_{\rho_{\epsilon}}(v)v =0\},
$$
$$
\Gamma_{\rho_{\epsilon}} = \{ \gamma \in C([0,1] , H_{\rho_{\epsilon}}^{\alpha}(\mathbb{R}^{n})):\;\; \gamma (0) = 0, \;\;I_{\rho_{\epsilon}}(\gamma (1)) <0 \}
$$
and the mountain pass minimax value
$$
C_{\rho_{\epsilon}} = \inf_{\gamma \in \Gamma_{\rho_{\epsilon}}} \max_{t \in [0,1]} I_{\rho_{\epsilon}}(\gamma (t)).
$$
From Lemma \ref{Glm3} we also have
\begin{equation}\label{C09}
0 < C_{\rho_{\epsilon}} = \inf_{v\in \mathcal{N}_{\rho_{\epsilon}}} I_{\rho_{\epsilon}}(v)= \inf_{v\in H_{\rho_{\epsilon}}^{\alpha}(\mathbb{R}^{n})\setminus \{0\}} \max_{t \geq 0} I_{\rho_{\epsilon}}(t v).
\end{equation}
For comparison purposes we consider the functional  $I$, whose critical points are the solutions of \equ{L01}. We also consider the critical value $C$ that satisfies 
$$
C=\inf_{u\in H^{\alpha}(\mathbb{R}^{n})\setminus \{0\}} \max_{t \geq 0} I(t u).
$$ 

Now we start the proof of  Theorem \ref{Itm2} with some preliminary lemmas.
\begin{Lem}\label{CClm02}
Suppose ($\rho_{1}$) holds. Then
\begin{equation}\label{C10}
\lim_{\epsilon \to 0^{+}} C_{\rho_{\epsilon}} = C.
\end{equation}
\end{Lem}
\begin{proof} Since we obviously have 
$$
\int_{\mathbb{R}^{n}}\int_{B(0,\rho_\epsilon (x))} \frac{|u(x+z) - u(x)|^{2}}{|z|^{n+2\alpha}}dzdx \leq \int_{\mathbb{R}^{n}}\int_{\mathbb{R}^{n}} \frac{|u(x+z) - u(x)|^{2}}{|z|^{n+2\alpha}}dzdx,
$$
for all $u\in H_{\rho_{\epsilon}}^{\alpha}(\mathbb{R}^{n})$, then   we have
$
I_{\rho_{\epsilon}} (u) \leq I(u)
$
and therefore
\begin{equation}\label{C11}
\limsup_{\epsilon \to 0^{+}} C_{\rho_{\epsilon}} \leq C.
\end{equation}
On the other hand, by ($\rho_1$) we have $\rho (\epsilon x) \geq \rho_0$ for all $x\in \mathbb{R}^n$ then 
$$
C_{\rho_\epsilon} \geq C_{\frac{\rho_0}{\epsilon}}
$$
By standard arguments we can show that 
$$
\lim_{\epsilon \to 0}C_{\frac{\rho_0}{\epsilon}} = C.
$$
Thus
\begin{equation}\label{C12}
\liminf_{\epsilon \to 0} C_{\rho_\epsilon} \geq C.
\end{equation}
Therefore, by (\ref{C11}) and (\ref{C12}) we obtain (\ref{C10}).
\end{proof}
\begin{Lem}\label{CClm03}
There are a $\epsilon_0>0$, a family $y_\epsilon \subset \mathbb{R}^n$, constants $\beta , R>0$ such that 
\begin{equation}\label{C13}
\int_{B(y_\epsilon , R)} v_{\epsilon}^2dx\geq \beta,\;\;\mbox{for all}\;\;\epsilon \in [0, \epsilon_0]. 
\end{equation}
\end{Lem}

\begin{proof}
By contradiction, there is a sequence $\epsilon_m \to 0$ such that for all $R>0$
$$
\lim_{m\to \infty} \sup_{y\in \mathbb{R}^n} \int_{B(y,R)} v_{\epsilon_m}^{2}dx = 0
$$
Using the following notation $v_m = v_{\epsilon_m}$ and $C_{\rho_m} = C_{\rho_{\epsilon_m}}$, by Lemma \ref{CClem}
$$
\int_{\mathbb{R}^n} v_m^{q+1}dx = o_{m}(1).
$$
Furthermore, since $I'_{\rho_m}(v_m)v_m = 0$ then 
$$
\int_{\mathbb{R}^n}\int_{B(0, \rho_m (x))}\frac{|v_m(x+z) - v_m(x)|^2}{|z|^{n+2\alpha}}dz dx + \int_{\mathbb{R}^n} v_m^2dx = \int_{\mathbb{R}^n} v_m^{2_{\alpha}^{*}}dx + o_m(1).
$$
Let $l\geq 0$ be such that 
$$
\int_{\mathbb{R}^n}\int_{B(0, \rho_m(x))} \frac{|v_m(x+z) - v_m(x)|^2}{|z|^{n+2\alpha}} dz dx + \int_{\mathbb{R}^n} v_m^2dx \to l.
$$
Now, since $I_{\rho_m}(v_m) = C_{\rho_m}$, we obtain 
$$
C_{\rho_m} = \frac{\alpha}{n}\left( \int_{\mathbb{R}^n}\int_{B(0, \rho_m(x))}\frac{|v_m(x+z) - v_m(x)|^2}{|z|^{n+2\alpha}}dz dx + \int_{\mathbb{R}^n}v_m^2dx \right) + o_m(1),
$$ 
then by Lemma \ref{CClm02}
\begin{equation}\label{C14}
l = \frac{n}{\alpha}C
\end{equation}
hence $l>0$: Now, using the definition of the Sobolev constant $S$ and Remark \ref{FSnta1}, we have
$$
\left( \int_{\mathbb{R}^n}v_m^{2_{\alpha}^{*}}dx \right)^{2/2_{\alpha}^{*}} S \leq \int_{\mathbb{R}^n}\int_{B(0, \rho_m(x))} \frac{|v_m(x+z) - v_m(x)|^2}{|z|^{n+2\alpha}}dz dx + \int_{\mathbb{R}^n}v_m^2dx
$$
Therefore, by (\ref{C14})  and taking the limit in the above inequality as $m\to \infty$ we achieved that
$$
C\geq \frac{\alpha}{n}S^{n/2\alpha}
$$
which is a contradiction with (\ref{L02}).
\end{proof}


Now let 
\begin{equation}\label{C15}
w_{\epsilon}(x) = v_{\epsilon}(x + y_{\epsilon}) = u_{\epsilon}(\epsilon x + \epsilon y_{\epsilon}),
\end{equation} 
then by (\ref{C13}),
\begin{equation}\label{C16}
\liminf_{\epsilon \to 0^{+}}\int_{B(0,R)} w_{\epsilon}^{2}(x)dx \geq \beta > 0.
\end{equation}
To continue, we consider the rescaled scope function $\overline\rho_\epsilon$, as defined in \equ{C04}, 
$$
\bar\rho_\epsilon(x)=\frac{1}{\epsilon}\rho(\epsilon x+\epsilon y_\epsilon)
$$
and then $w_\epsilon$ satisfies the equation 
\begin{equation}\label{C17}
(-\Delta)_{\overline{\rho}_{\epsilon}}^{\alpha}w_{\epsilon}(x) + w_{\epsilon}(x) = w_{\epsilon}^{p}(x), \;\;\mbox{in}\;\;\mathbb{R}^{n}.
\end{equation}

Now we   prove  the convergence of $w_\epsilon$ as $\epsilon\to 0$.
\begin{Lem}\label{CClm04} 
For every sequence $\{\epsilon_m\}$ there is a subsequence, we keep calling the same, so
that 
$w_{\epsilon_m}=w_{m} \to w$ in $H^{\alpha}(\mathbb{R}^{n})$, when $m\to \infty$, where $w$ is a solution of \equ{L01}.
\end{Lem}
\begin{proof}
Note that
$$
\begin{aligned}
C_{\rho_m} &= I_{\rho_m}(v_m) - \frac{1}{q+1}I'_{\rho_m}(v_m)v_m\\
&= \left( \frac{1}{2}-\frac{1}{q+1} \right)\left( \int_{\mathbb{R}^n}\int_{B(0, \frac{1}{\epsilon_m}\rho(\epsilon_m x))}\frac{|v_m(x+z) - v_m(x)|^2}{|z|^{n+2\alpha}}dz dx + \int_{\mathbb{R}^n}v_m^2(x)dx \right)\\
& + \left( \frac{1}{q+1} - \frac{1}{2_{\alpha}^{*}} \right) \int_{\mathbb{R}^n}v_{m}^{2_{\alpha}^{*}}(x)dx\\
&\geq \left( \frac{1}{2}-\frac{1}{q+1} \right)\left( \int_{\mathbb{R}^n}\int_{B(0, \frac{\rho_0}{\epsilon_m})}\frac{|v_m(x+z) - v_m(x)|^2}{|z|^{n+2\alpha}}dz dx + \int_{\mathbb{R}^n}v_m^2(x)dx \right)\\
& + \left( \frac{1}{q+1} - \frac{1}{2_{\alpha}^{*}} \right) \int_{\mathbb{R}^n}v_{m}^{2_{\alpha}^{*}}(x)dx = \Lambda_m
\end{aligned}
$$
By Lemma \ref{CClm02} we obtain that 
\begin{equation}\label{C18}
\limsup_{m\to \infty} \Lambda_m \leq C.
\end{equation}
On the other hand, by Fatou's Lemma and the weak convergence of $\{w_m\}$, we get 
\begin{equation}\label{C19}
\begin{aligned}
C&\leq I(w)\\
& = \left( \frac{1}{2} - \frac{1}{q+1} \right)\left( \int_{\mathbb{R}^n}\int_{\mathbb{R}^n} \frac{|w(x+z) - w(x)|^2}{|z|^{n+2\alpha}} dz dx + \int_{\mathbb{R}^n}w^2dx \right) + \left( \frac{1}{q+1}-\frac{1}{2_{\alpha}^{*}} \right)\int_{\mathbb{R}^n}w^{2_{\alpha}^{*}}dx\\
&\leq \liminf_{m\to \infty} \Lambda_m + \liminf_{m\to \infty} \left( \frac{1}{2}-\frac{1}{q+1} \right)\int_{\mathbb{R}^n}\int_{\mathbb{R}^n \setminus B(0, \frac{\rho_0}{\epsilon_m})} \frac{|w_m(x+z) - w_m(x)|^2}{|z|^{n+2\alpha}}dz dx\\
&= \liminf_{m\to \infty} \Lambda_m. 
\end{aligned}
\end{equation}
So, by (\ref{C18}) and (\ref{C19}), $\lim_{m\to \infty}\Lambda_m = C$, from where we get
$$
\lim_{m\to \infty} \|w_m - w\|_{\alpha} = 0.
$$

\end{proof}
\smallskip

We are now in a position to complete the proof of our second main theorem.

\smallskip

\noindent
{\bf Proof of Theorem \ref{}} We first obtain an upper bound for the critical values $C_{\rho_{\epsilon_m}}=C_m$, for the sequence $\{\epsilon_m\}$ given in Lemma \ref{CClm04}.
Next we consider the scope function
$$
\tilde\rho_m(x)=\frac{1}{\epsilon_m}\rho(\epsilon_m x+x_0),
$$ 
where  $x_0$ is a global minimum point of $\mathcal{H}$, see Lemma \ref{CClm01}. To continue, we consider the function $w_m=w_{\epsilon_m}$ as given in \equ{C15} and  let $t_m>0$ such that $t_m w_m\in \mathcal{N}_{\tilde\rho_m}$. According 
to Lemma \ref{CClm04}, $\{w_m\}$ converges to $w\in\mathcal{N}$, then $t_m\to 1$ and $t_mw_m\to w$.

 Now we apply Theorem \ref{CCtm01} to obtain that
\begin{equation}\label{comp1}
C_m\le I_{\tilde\rho_m}(t_m w_m)=I_{\frac{\rho_\infty}{\epsilon_m}}(t_{m}{w}_{m}) + \epsilon_{m}^{2\alpha}\left(  \|{w}\|_{L^{2}}^{2}\mathcal{H}(x_{0}) + o(1)\right).
\end{equation}
We have used part (i) of Theorem \ref{CCtm01} with $z_m=x_0/{\epsilon_m}$.

On the other hand, since  $w_{m} \in H^{\alpha}(\mathbb{R}^{n})$ is a critical point of $I_{\bar\rho_m}$, we have that
\begin{equation}\label{comp2}
C_m=I_{\bar\rho_m}(w_m)\ge I_{\bar\rho_m}(t_mw_m).
\end{equation}
We write $y_m=y_{\epsilon_m}$. If $\epsilon_m|y_m|\to \infty$,
then we may apply part (ii) of Theorem \ref{CCtm01} with $z_m=y_m$ in \equ{comp2} and obtain that
$$
C_m\ge I_{\frac{\rho_\infty}{\epsilon_m}}(t_{m}{w}_{m}) + \epsilon_{m}^{2\alpha} o(1),
$$
which contradicts \equ{comp1}. We conclude then, that $\{\epsilon_my_m\}$ is bounded and that, for a subsequence, $\epsilon_my_m\to \bar x$, for some $\bar x\in \R^n$. Now we apply Theorem \ref{CCtm01} again, but now part (i) with $z_m=y_m$ in \equ{comp2}, and we obtain that
\begin{equation}\label{comp3}
C_m\ge I_{\frac{\rho_\infty}{\epsilon_m}}(t_{m}{w}_{m}) + \epsilon_{m}^{2\alpha}\left(  \|{w}\|_{L^{2}}^{2}\mathcal{H}(\bar x) + o(1)\right).
\end{equation}
From \equ{comp1} and \equ{comp3} we finally get that 
$$
\|w\|_{L^{2}}^{2}\mathcal{H}(\overline{x}) + o(1) \leq \|w\|_{L^{2}}^{2} \mathcal{H}(x_{0}) + o(1)
$$
and taking  the limit as $m\to\infty$, we get
\begin{equation}\label{Ceq38}
\mathcal{H}(\overline{x}) \leq \mathcal{H}(x_{0})
\end{equation}
completing the proof of the theorem. $\Box$

\end{document}